\documentclass[reqno]{amsart}

\usepackage[english]{babel}
\usepackage{amsthm, latexsym, enumerate, gastex, amssymb,url,listings,enumitem}

\usepackage[mathcal]{eucal}
\usepackage{graphicx,url}

\copyrightinfo{2015}{Simone Ugolini}

\renewcommand{\phi}[0]{\varphi}
\renewcommand{\theta}[0]{\vartheta}
\renewcommand{\epsilon}[0]{\varepsilon}

\newcommand{\Z}{\text{$\mathbf{Z}$}}

\newcommand{\Pro}{\text{$\mathbf{P}^1$}}
\newcommand{\F}{\text{$\mathbf{F}$}}

\newtheorem{theorem}{Theorem}[section]
\newtheorem{lemma}[theorem]{Lemma}
\newtheorem{corollary}[theorem]{Corollary}

\theoremstyle{definition}

\newtheorem{example}[theorem]{Example}

\theoremstyle{remark}
\newtheorem{remark}[theorem]{Remark}

\numberwithin{equation}{section}

\begin{document}

\bibliographystyle{amsplain}

\date{}

\title[]
{Sequences of irreducible polynomials over odd prime fields via elliptic curve endomorphisms, II}

\author{S.~Ugolini}
\email{sugolini@gmail.com} 

\begin{abstract}
In this paper we extend a previous investigation by us regarding an iterative construction of irreducible polynomials over finite fields of odd characteristic. In particular,  we show how it is possible to iteratively construct irreducible polynomials by means of two families of transforms,  which we call the $Q_k$ and $\hat{Q}_k$-transforms, related to certain degree two isogenies over elliptic curves, which split the multiplication-by-$2$ map.  
\end{abstract}

\maketitle

\section{Introduction}
Inspired by the $Q$-transform and the $R$-transform (see \cite{coh}, \cite{mey}), in \cite{SOPVE} we defined the $Q_k$-transforms over any finite field of odd characteristic as follows. 

If $p$ is an odd prime, $q$ is a power of $p$ and $k \in \F_p^*$, then the $Q_k$-transform takes any polynomial $f \in \F_p [x]$ of positive degree $n$ to 
\begin{equation*}
f^{Q_k} (x) = \left( \frac{x}{k} \right)^n \cdot f(\theta_{k}(x)),
\end{equation*} 
where $\theta_{k}$ is the map which takes any element $x \in \Pro (\F_q)  = \F_q \cup \{ \infty \}$ to
\begin{displaymath}
\theta_{k} (x) = 
\begin{cases}
\infty & \text{if $x= 0$ or $\infty$,}\\
k \cdot (x+x^{-1}) & \text{otherwise}.
\end{cases}
\end{displaymath}

In \cite{SOPVE} we showed how one can construct sequences of irreducible polynomials over finite fields by repeated applications of the $Q_k$-transforms when $k$ and $p$ fall into one of the following cases:
\begin{itemize}
\item $k \equiv \pm  \frac{1}{2} \mod p$;
\item $k $ is a root of $x^2 + \frac{1}{4}$ and $p \equiv 1 \pmod{4}$;
\item $k$ or $-k$  is a root of $x^2 + \frac{1}{2} x + \frac{1}{2}$ and $p \equiv 1, 2$, or $4 \pmod{7}$.
\end{itemize}

Since the dynamics of the maps $\theta_k$ seems to be chaotic for any $k$ different from the aforementioned values, in this paper we illustrate an iterative construction of irreducible polynomials, which is independent of the characteristic of the field and employs two families of transforms, namely the $Q_k$-transforms and the $\hat{Q}_k$-transforms, which are below introduced.

We notice that if $k$ is a quadratic residue in $\F_p^*$ and $\alpha_k$ is a square root of $k^3$ in $\F_p^*$, then the map $\theta_k$ is involved in the definition of the isogeny
\begin{eqnarray*}
\psi_k (x,y)  & = &  \left(\theta_k(x),  \alpha_k \cdot \frac{x^2 y - y}{x^2} \right)
\end{eqnarray*}
from the elliptic curve 
\begin{displaymath}
\begin{array}{lll}
E & : &  y^2 = x^3 + x
\end{array}
\end{displaymath}
to the elliptic curve
\begin{displaymath}
\begin{array}{lll}
{E}_k & : & y^2 = x^3 - 4 k^2 x.
\end{array}
\end{displaymath} 

Consider now the map $\hat{\theta}_k$ which takes any element $x \in \Pro (\F_q) $ to
\begin{displaymath}
\hat{\theta}_{k} (x) = 
\begin{cases}
\infty & \text{if $x= 0$ or $\infty$,}\\
\frac{x^2-4k^2}{4kx} & \text{otherwise,}
\end{cases}
\end{displaymath}
and the $\hat{Q}_k$-transform, which takes any polynomial $f \in \F_p [x]$ of positive degree $n$ to 
\begin{equation*}
f^{\hat{Q}_k} (x) = \left( 4 k x \right)^n \cdot f(\hat{\theta}_{k}(x)).
\end{equation*} 
The map $\hat{\theta}_{k}$ is involved in the definition of the isogeny
\begin{eqnarray*}
\hat{\psi}_k (x,y) & = & \left(\hat{\theta}_k (x),  \frac{y(x^2+4k^2)}{8 \alpha_k x^2} \right)
\end{eqnarray*}
from $E_k$ to $E$, namely the dual isogeny of $\psi_k$.
If we denote by $[2]$ the duplication map on $E$, then 
\begin{displaymath} 
\begin{array}{lll}
[2] & = &  \hat{\psi}_k \circ \psi_k.
\end{array}
\end{displaymath}

While the isogenies $\psi_k$ and $\hat{\psi}_k$ have been defined only for the quadratic residues $k$ in $\F_p^*$, the construction of sequences of irreducible polynomials, which is described in Section \ref{sec_const}, can be carried over to any $k \in \F_p^*$, as explained in Remark \ref{any_k_is_ok}.

\section{Preliminaries}\label{sec_prel}
Let $\F_q$ be a finite field of odd characteristic $p$. 

The structure of the group $E(\F_{q})$ of rational points  of $E$ over $\F_{q}$ depends upon $p$. In fact, if $p \equiv 1 \pmod{4}$, then $E$ is an ordinary elliptic curve, while $E$ is supersingular if $p \equiv 3 \pmod{4}$ (see \cite[Proposition 4.37]{was}).

Whichever $p$ is, we can  consider the map $[\tilde{2}]$ defined over $\Pro(\overline{\F}_q)$  as
\begin{displaymath}
[\tilde{2}] : x \mapsto 
\begin{cases}
\infty & \text{if $x \in \{0, i_p, -i_p, \infty \}$,}\\
\frac{x^4-2x^2+1}{4(x^3+x)} & \text{otherwise,}
\end{cases}
\end{displaymath}
where $i_p$ is a square root of $-1$ in $\overline{\F}_p$.

For any $x \in \Pro (\overline{\F}_q)$ and any quadratic residue $k \in \F_p^*$ we have that
\begin{equation}\label{2_is_theta}
[\tilde{2}] (x) =  \hat{\theta}_k (\theta_k(x)).
\end{equation}

\begin{remark}\label{any_k_is_ok}
While in the current section $k$ is assumed to be a quadratic residue in $\F_p^*$, we notice that $[\tilde{2}] = \hat{\theta}_k \circ \theta_k$ whichever $k \in \F_p^*$ we take. This fact will let us to extend our iterative construction of irreducible polynomials in Section \ref{sec_const} to any $k$.
\end{remark}

We can construct the functional graph $G^{q}_{[\tilde{2}]}$ of $[\tilde{2}]$ over $\Pro (\F_q)$, where the vertices are the elements of $\Pro (\F_q)$ and an arrow joins a vertex $\alpha$ to a vertex $\beta$ if $\beta =  [\tilde{2}] (\alpha)$. Since any vertex of $G^{q}_{[\tilde{2}]}$ is either $[\tilde{2}]$-periodic or preperiodic, any connected  component  of $G^{q}_{[\tilde{2}]}$ contains exactly one cycle, whose vertices are roots of reversed trees.
In the following, for any non-negative integer $i$ and any $[\tilde{2}]$-periodic element $x_0 \in \Pro(\F_q)$, we denote by $T^{q^{2^i}}_{[\tilde{2}]} (x_0)$ the reversed tree of $G^{q^{2^i}}_{[\tilde{2}]}$ rooted in $x_0$. 

The following holds.

\begin{lemma}\label{counterimages}
If $\tilde{x} \in \Pro (\overline{\F}_q)$, then
\begin{displaymath}
|\{x \in \Pro(\overline{\F}_q) : [\tilde{2}] (x) = \tilde{x} \}| =
\begin{cases}
2 & \text{if $\tilde{x} \in \{\pm i_p, 0 \}$},\\
4 & \text{otherwise.}
\end{cases}
\end{displaymath}
\end{lemma}
\begin{proof}
We introduce the following notations for any $x_0 \in \Pro(\overline{\F}_q)$:
\begin{eqnarray*}
\theta^{-1}_k \{ x_0 \} & = & \{x \in \Pro(\overline{\F}_q) : \theta_k (x) = x_0 \};\\
\hat{\theta}^{-1}_k \{ x_0 \} & = & \{x \in \Pro(\overline{\F}_q) : \hat{\theta}_k (x) = x_0 \};\\
\left[\tilde{2} \right]^{-1} \{ x_0 \} & = & \{x \in \Pro(\overline{\F}_q) : [\tilde{2}] (x) = x_0 \}.
\end{eqnarray*}

We have that
\begin{displaymath}
\begin{array}{rcl}
|\theta^{-1}_k \{ x_0 \}| & = &
\begin{cases}
1 & \text{if $x_0 \in \{ \pm 2k \}$,}\\
2 & \text{otherwise,}
\end{cases}
\\[4ex]
|\hat{\theta}^{-1}_k \{ x_0 \}| & = &
\begin{cases}
1 & \text{if $x_0 \in \{ \pm i_p \}$,}\\
2 & \text{otherwise.}
\end{cases}
\end{array}
\end{displaymath}

Moreover,
\begin{displaymath}
\begin{array}{rcl}
\hat{\theta}_k \{ \pm 2 k i_p \} & = & \{ \pm i_p \},\\
\hat{\theta}_k \{ \pm 2 k \} & = & \{ 0 \}.
\end{array}
\end{displaymath}

We can now analyse the different cases.
\begin{itemize}
\item If $\tilde{x} \in \Pro(\F_q) \backslash \{ \pm i_p, 0 \}$, then $\hat{\theta}^{-1}_k \{ \tilde{x} \} = \{x_1, x_2\}$,
where $\{x_1, x_2\} \cap \{\pm 2 k \} = \emptyset$. Therefore, $\left| \left[\tilde{2} \right]^{-1} \{ \tilde{x} \} \right| =4$. 
\item If  $\tilde{x} \in \{\pm i_p \}$, then $\hat{\theta}^{-1}_k \{ \tilde{x} \} = \{ 2 k \tilde{x} \}$. Therefore, $\left| \left[\tilde{2} \right]^{-1} \{ \tilde{x} \} \right| = 2$. 
\item If  $\tilde{x} = 0$, then $\hat{\theta}^{-1}_k \{ \tilde{x} \} = \{ \pm 2 k \}$. Therefore, $\left| \left[\tilde{2} \right]^{-1} \{ \tilde{x} \} \right| = 2$.
\end{itemize}

All considered, the result follows.
\end{proof}

According to Lemma \ref{counterimages} the following holds. 
\begin{corollary}
Let $x_0 \in \Pro (\F_q)$ be $[\tilde{2}]$-periodic. Then, $T^{q^{2^i}}_{[\tilde{2}]} (x_0)$ is a $4$-ary tree for any non-negative integer $i$.
\end{corollary}

\begin{example}\label{exm_49}
Below  is represented the graph $G_{[\tilde{2}]}^{49}$. As regards the labels of the nodes, `0' is the zero in $\F_{49}$, while all the other labels different from $\infty$ refer to the exponents of the powers $\alpha^i$, being $\alpha$ a generator  of the field $\F_{49}$. We notice that every node, which is not a leaf, has exactly $4$ children, except for $12$, $36$ and `0'. This fact is in accordance with Lemma \ref{counterimages}, since $(\alpha^{12})^2 = (\alpha^{36})^{2} = -1$.

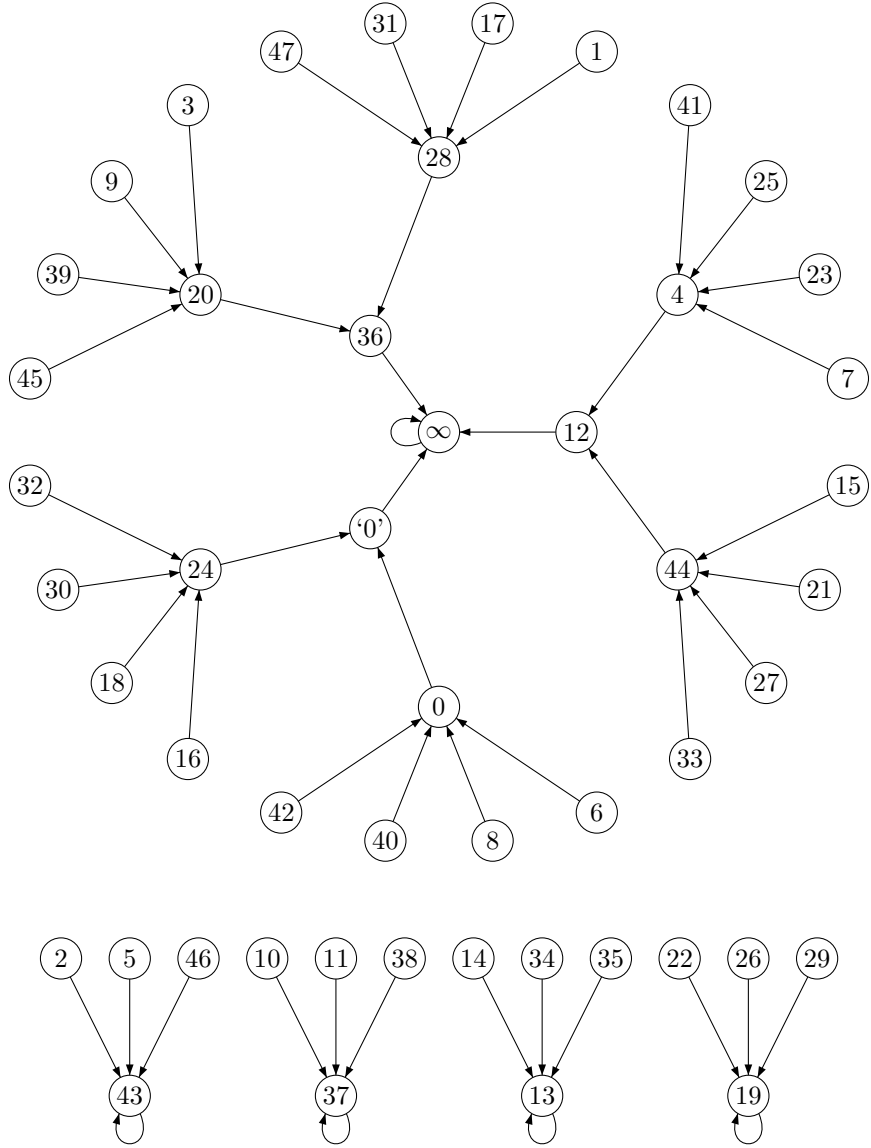
\begin{center}
    \unitlength=2.6pt
    \begin{picture}(130, 130)(-65,-65)
    \gasset{Nw=6,Nh=6,Nmr=3,curvedepth=0}
    \thinlines
   
	\node(O)(0,0){$\infty$}   
   
    \node(A1)(20,0){$12$}
    \node(A2)(-10,14){$36$}
    \node(A3)(-10,-14){`0'}
    
    \node(A10)(34.7,20){$4$}
    \node(A20)(0,40){$28$}
    \node(A21)(-34.7,20){$20$}  
    \node(A30)(-34.7,-20){$24$}
    \node(A31)(0,-40){$0$}
    \node(A11)(34.7,-20){$44$}

    \drawedge(A1,O){}
    \drawedge(A2,O){}
    \drawedge(A3,O){}
    
    \node(A100)(59.49,7.8){$7$}
    \node(A101)(55.4,22.96){$23$}
    \node(A102)(47.6,36.53){$25$}  
    \node(A103)(36.53,47.6){$41$}
    
    \node(A200)(22.96,55.4){$1$}
    \node(A201)(7.8,59.49){$17$}
    \node(A202)(-7.8,59.49){$31$}  
    \node(A203)(-22.96,55.4){$47$}
    
    \node(A210)(-36.53,47.6){$3$}
    \node(A211)(-47.6,36.53){$9$}
    \node(A212)(-55.4,22.96){$39$}  
    \node(A213)(-59.49,7.8){$45$}
    
    \node(A300)(-36.53,-47.6){$16$}
    \node(A301)(-47.6,-36.53){$18$}
    \node(A302)(-55.4,-22.96){$30$}  
    \node(A303)(-59.49,-7.8){$32$}
    
    \node(A310)(22.96,-55.4){$6$}
    \node(A311)(7.8,-59.49){$8$}
    \node(A312)(-7.8,-59.49){$40$}  
    \node(A313)(-22.96,-55.4){$42$}
    
    \node(A110)(59.49,-7.8){$15$}
    \node(A111)(55.4,-22.96){$21$}
    \node(A112)(47.6,-36.53){$27$}  
    \node(A113)(36.53,-47.6){$33$}

    \drawedge(A10,A1){}
    \drawedge(A11,A1){}
    \drawedge(A20,A2){}
    \drawedge(A21,A2){}
    \drawedge(A30,A3){}
    \drawedge(A31,A3){}
    
    \drawedge(A100,A10){}
    \drawedge(A101,A10){}
    \drawedge(A102,A10){}
    \drawedge(A103,A10){}
    \drawedge(A110,A11){}
    \drawedge(A111,A11){}
    \drawedge(A112,A11){}
    \drawedge(A113,A11){}
    \drawedge(A200,A20){}
    \drawedge(A201,A20){}
    \drawedge(A202,A20){}
    \drawedge(A203,A20){}
    \drawedge(A210,A21){}
    \drawedge(A211,A21){}
    \drawedge(A212,A21){}
    \drawedge(A213,A21){}
    \drawedge(A300,A30){}
    \drawedge(A301,A30){}
    \drawedge(A302,A30){}
    \drawedge(A303,A30){}
    \drawedge(A310,A31){}
    \drawedge(A311,A31){}
    \drawedge(A312,A31){}
    \drawedge(A313,A31){}
    
    \drawloop[loopangle=180,loopdiam=4](O){}

    \end{picture}
  \end{center}
  
  \begin{center}
    \unitlength=2.6pt
    \begin{picture}(110, 40)(0,-10)
    \gasset{Nw=6,Nh=6,Nmr=3,curvedepth=0}
    \thinlines

    \node(A1)(10,0){$43$}
    \node(A2)(40,0){$37$}
    \node(A3)(70,0){$13$}
    \node(A4)(100,0){$19$}
    
    \node(A10)(0,20){$2$}
    \node(A11)(10,20){$5$}
    \node(A12)(20,20){$46$}
    
    \node(A20)(30,20){$10$}
    \node(A21)(40,20){$11$}
    \node(A22)(50,20){$38$}   
    
    \node(A30)(60,20){$14$}
    \node(A31)(70,20){$34$}
    \node(A32)(80,20){$35$}
    
    \node(A40)(90,20){$22$}
    \node(A41)(100,20){$26$}
    \node(A42)(110,20){$29$} 

    \drawedge(A10,A1){}
    \drawedge(A11,A1){}
    \drawedge(A12,A1){}
    \drawedge(A20,A2){}
    \drawedge(A21,A2){}
    \drawedge(A22,A2){}
    \drawedge(A30,A3){}
    \drawedge(A31,A3){}
	\drawedge(A32,A3){}
	\drawedge(A40,A4){}
    \drawedge(A41,A4){}
	\drawedge(A42,A4){}

    \drawloop[loopangle=-90,loopdiam=4](A1){}
    \drawloop[loopangle=-90,loopdiam=4](A2){}
    \drawloop[loopangle=-90,loopdiam=4](A3){}
    \drawloop[loopangle=-90,loopdiam=4](A4){}
    \end{picture}
  \end{center}
\end{example}

\subsection{The ordinary case}\label{prel_ord}
In \cite[Section 3]{SUk} and in \cite[Section 2.1]{SOPVE}, relying upon \cite{wit}, we studied some properties of the group of rational points of $E$ over a finite field. We summarize the relevant facts for the reader's convenience.

Let $m$ be a positive integer, $l$ a non-negative integer and $q = p^{2^l m}$. If we set $R = \Z[i]$ and denote by $\pi_p$ the representation in $R$ of the Frobenius endomorphism of $E$, then
\begin{equation*}
E(\F_{p^{2^l m}}) \cong R / (\pi_p^{2^l m}-1) R \cong   R / \rho_0^{e_l} R \times R / \rho_1 R,
\end{equation*}
where $\rho_0= 1 + i$, $e_l$ is a non-negative integer which depends on $l$ and $\rho_1$ is an element of $R$ coprime to $\rho_0$ such that $\rho_0^{e_l} \cdot \rho_1 = \pi_p^{2^l m}-1$. According to \cite[Lemma 2.13 (1), (4)]{SOPVE}, the following holds.
\begin{lemma}\label{l_l+1}
We have that
\begin{itemize}
\item $e_l \geq 2$; 
\item $e_l = e_{l-1}+2$, if $l \geq 2$.
\end{itemize}
\end{lemma}

Since $[2] = \hat{\psi}_k \circ \psi_k$ and $2 = - i \cdot \rho_0^2$ in $R$, we can prove the forthcoming result concerning the depth of the trees rooted in $[\tilde{2}]$-periodic elements of $G^q_{[\tilde{2}]}$.

\begin{theorem}\label{t_l+1}
Let $x_0 \in \Pro (\F_q)$ be $[\tilde{2}]$-periodic. Then,
\begin{enumerate}
\item $T^{q^{2}}_{[\tilde{2}]} (x_0)$ has depth $d := \lceil \frac{e_{l+1}}{2} \rceil$ and its leaves have height  at least $d - 1$;
\item the children of the leaves of $T^{q^{2^i}}_{[\tilde{2}]} (x_0)$ in $\Pro (\overline{\F}_q)$ are leaves of $T^{q^{2^{i+1}}}_{[\tilde{2}]} (x_0)$, for any positive integer $i$. 
\end{enumerate}
\end{theorem} 
\begin{proof}
\begin{enumerate}[leftmargin=*]
\item 
The dynamics of $[\tilde{2}]$ over $\Pro (\F_{q^2})$ can be studied relying upon the iterations of  $[2] = [- i  \rho_0^2]$ in 
\begin{equation*}
R / (\pi_p^{2^{l+1}m}-1) R \cong S =   R / \rho_0^{e_{l+1}} R \times R / \rho_1 R.
\end{equation*}

By hypothesis, $x_0 \in \Pro(\F_q)$. Therefore, either $x_0 =  \infty$ or $(x_0, y_0) \in E (\F_{q^2})$, for some $y_0 \in \F_{q^2}$. In both cases, the corresponding point $Q$ in $S$ is of the form $Q = (0, Q_1)$. 

Consider the point $P = ([1], [2]^{-d}  Q_1) \in S$. Then, $[2]^{d} P = Q$, while $[2]^{h} P \not = Q$ for any positive integer $h < d$. 
More in general, if $(P_0, P_1) \in S$, then $[2]^{d} P_0 = 0$ in $R / \rho_0^{e_{l+1}} R$. Hence, $T^{q^{2}}_{[\tilde{2}]} (x_0)$ has depth $d$.

Let now $\tilde{x}$ be a leaf of $T^{q^{2}}_{[\tilde{2}]} (x_0)$. Suppose that $P = (P_0, P_1)$ is the point in $S$ having such a $x$-coordinate and that $P_0 = [a]$ for some $a \in R$. Then, $\rho_0^2 \nmid a$. Indeed, if $a = \rho_0^2 c$ for some $c \in R$, then we could take the point $\tilde{P} = ([i c] , [2]^{-1} P_1)$ and notice that $[2] \tilde{P} = P$, which is absurd, since $\tilde{x}$ is a leaf of the tree.  Consequently, if $[2]^h P_0 = 0$ for some positive integer $h$, then $h \geq \lceil \frac{e_{l+1}}{2} \rceil - 1$.

\item Consider a leaf $\tilde{x}$ of $T^{q^{2^i}}_{[\tilde{2}]} (x_0)$, for some positive integer $i$. Let $x'$ be one of the direct predecessors of $\tilde{x}$ in $T^{q^{2^{i+1}}}_{[\tilde{2}]} (x_0)$. Since the greatest power of $\rho_0$ which divides $\pi_p^{2^{l+i+1} m}-1$ is $e_{l+i+1}$ and
$e_{l+i+1} = e_{l+i}+2$ according to Lemma \ref{l_l+1}, we have that $x'$ is a leaf of $T^{q^{2^{i+1}}}_{[\tilde{2}]} (x_0)$.
\end{enumerate}
\end{proof}

\subsection{The supersingular case}
Let $i$ and $m$ be two positive integers. Then, according to \cite[Theorem 4.1]{wit},
\begin{equation*}
E(\F_{p^{2^{i} m}}) \cong \Z / ((-p)^{2^{i-1} m}-1) \Z \times  \Z / ((-p)^{2^{i-1} m}-1) \Z.
\end{equation*}
The following holds.

\begin{lemma}\label{super_s_lemma}
There exist two positive integers $e_{i-1}$ and $e_{i}$ and two odd  integers $r$ and $s$ such that
\begin{displaymath}
\begin{array}{lcl}
(-p)^{2^{i-1} m}-1 & = & 2^{e_{i-1}} \cdot r,\\
(-p)^{2^{i} m}-1 & = & 2^{e_{i}} \cdot s.
\end{array}
\end{displaymath}
Moreover, $e_{i} = e_{i-1} +1$.
\end{lemma}
\begin{proof}
Since $-p \equiv 1 \pmod{4}$, we have that
\begin{eqnarray*}
(-p)^{2^{i-1} m}  - 1 & \equiv & 0 \pmod{4},\\
(-p)^{2^{i-1} m} + 1 & \equiv & 2 \pmod{4}.
\end{eqnarray*}
Therefore, 
\begin{eqnarray*}
(-p)^{2^{i-1} m}-1  & = & 2^{e_{i-1}} \cdot r,\\
(-p)^{2^{i-1} m}+1  & = & 2 \cdot r',
\end{eqnarray*}
for some integer $e_{i-1} \geq 2$ and some odd integers $r$ and $r'$. Hence,
\begin{equation*}
(-p)^{2^{i} m}-1 = ((-p)^{2^{i-1} m}-1) \cdot ((-p)^{2^{i-1} m}+1) = 2^{e_{i-1}+1} \cdot r \cdot r'.
\end{equation*}
The result follows setting $e_{i} = e_{i-1}+1$ and $s = r \cdot r'$.
\end{proof}

According to Lemma \ref{super_s_lemma}, if we set $q = p^{m}$, then
\begin{equation*}
E(\F_{q^{2^{i}}}) \cong S_{i} = (\Z / 2^{e_{i-1}} \Z \times \Z / r \Z)^2. 
\end{equation*}
The following holds.

\begin{theorem}\label{t_l+1_super}
Let $x_0 \in \Pro (\F_q)$ be $[\tilde{2}]$-periodic. Then 
\begin{enumerate}
\item $T^{q^{2^i}}_{[\tilde{2}]} (x_0)$ has depth $e_{i-1}$;
\item the children of the leaves of $T^{q^{2^i}}_{[\tilde{2}]} (x_0)$ in $\Pro (\overline{\F}_q)$ are leaves of $T^{q^{2^{i+1}}}_{[\tilde{2}]} (x_0)$. 
\end{enumerate}
\end{theorem} 
\begin{proof}
\begin{enumerate}[leftmargin=*]
\item
Since $x_0$ is $[\tilde{2}]$-periodic in $\Pro (\F_q)$, it is the $x$-coordinate of a rational point in $E(\F_{q^{2^i}})$, which corresponds to a point in $S_i$ of the form
\begin{equation*}
([0], [a_r], [0], [b_r])
\end{equation*} 
for some integers $a_r$ and $b_r$. We notice that $e_{i-1}$ is the smallest positive integer $k$ such that $[2]^k [c] = [0]$ for any $[c]$ in $\Z / 2^{e_{i-1}} \Z$. Indeed, any leaf of $T_{[\tilde{2}]}^{q^{2^i}} (x_0)$ is the $x$-coordinate of a point 
\begin{equation*}
([a_2], [2]^{-e_{i-1}}[a_r], [b_2], [2]^{-e_{i-1}} [b_r])
\end{equation*} 
in $S_{i}$ for some integers $a_2$ and $b_2$ which are not both divisible by $2$ in $\Z$. Therefore, the tree $T^{q^{2^i}}_{[\tilde{2}]} (x_0)$ has depth $e_{i-1}$ and, in analogy, $T^{q^{2^{i+1}}}_{[\tilde{2}]} (x_0)$ has depth $e_{i}$.
\item Any leaf of $T^{q^{2^i}}_{[\tilde{2}]} (x_0)$ lies on the level $e_{i-1}$ of the tree $T^{q^{2^{i+1}}}_{[\tilde{2}]} (x_0)$, which has depth $e_i$. Consequently, its children are leaves of $T^{q^{2^{i+1}}}_{[\tilde{2}]} (x_0)$.
\end{enumerate}
\end{proof}

\begin{example}
Let $q = 7^2$. Then,
\begin{equation*}
E(\F_{7^{2}}) \cong  \Z / (-8) \Z \times \Z / (-8) \Z. 
\end{equation*}
According to Theorem \ref{t_l+1_super}, any $[\tilde{2}]$-periodic element $x_0 \in \Pro (\F_{49})$ is root of a tree having depth $3$. This is the case of $\infty$, as we can see in Example \ref{exm_49}.
\end{example}

\section{Constructing sequences of irreducible polynomials}\label{sec_const}
Let $f$ be a monic irreducible polynomial of positive degree $n$ belonging to $\F_p[x]$, for some odd prime $p$, and set $q=p^n$. For a fixed $k \in \F_p^*$ we can construct two sequences $\{g_i \}_{i \geq 0}$ and $\{h_i \}_{i \geq 0}$ of monic irreducible polynomials as follows. 
\begin{itemize}
\item We set $g_0:=f$ and $h_0:=f$. 
\item We set $g_1 := f^{\hat{Q}_k}$ and $h_1 := f^{\hat{Q}_k}$, if $f^{\hat{Q}_k}$ is irreducible. Otherwise, we set $g_1$ equal to one of the two monic irreducible factors of $f^{\hat{Q}_k}$ and $h_1$ equal to the other factor. 
\item For any positive integer $i$ we set 
$g_i$ (resp. $h_i$) equal to 
\begin{itemize}
\item one of the monic irreducible factors of  $g_{i-1}^{\hat{Q}_k}$ (resp.  $h_{i-1}^{\hat{Q}_k}$), if $i$ is odd;
\item one of the monic irreducible factors of  $g_{i-1}^{Q_k}$ (resp.  $h_{i-1}^{{Q}_k}$), if $i$ is even. 
\end{itemize} 
\end{itemize}  
\begin{remark}
We notice in passing that if $\tilde{f} \in \F_p [x]$ is irreducible of degree $m$, then either $\tilde{f}^{Q_k}$ (resp. $\tilde{f}^{\hat{Q}_k}$) is  irreducible of degree $2m$, or it splits into the product of two irreducible factors of degree $m$. Indeed, if $\alpha$ is a root of  $\tilde{f}^{Q_k}$ (resp. $\tilde{f}^{\hat{Q}_k}$), then $f(\theta_k(\alpha))= 0 $ (resp. $f(\hat{\theta}_k (\alpha))=0$). Hence, either $\alpha$ has degree $2m$ or it has degree $m$ over $\F_p$. 
\end{remark}

The following holds.
\begin{lemma}\label{lemma_seq}
If $\tilde{x}$ is a root of $f$ in $\F_q$, then
\begin{enumerate}
\item $\tilde{x}$ belongs to the level $r$ of the tree $T^q_{[\tilde{2}]} (x_0)$, for some non-negative integer $r$ and $x_0 \in \Pro (\F_q)$;
\item for any positive integer $j$, either any polynomial $g_{2j}$ or any polynomial $h_{2j}$ has a root $\tilde{x}_j$ belonging to the level $r+j$ of the tree $T^{q^{2^i}}_{[\tilde{2}]} (x_0)$ for some non-negative integer $i$.
\end{enumerate}
\end{lemma}
\begin{proof}
We prove separately the two assertions.
\begin{enumerate}[leftmargin=*]
\item The assertion holds because any element in $\Pro(\F_q)$ is either $[\tilde{2}]$-periodic or preperiodic. In the former case $x_0 = \tilde{x}$, while in the latter case some iterate of $\tilde{x}$ is $[\tilde{2}]$-periodic and we set $x_0$ equal to the first of such iterates which is  $[\tilde{2}]$-periodic.
\item The assertion can be proved by induction on $j$. 

First we define $\tilde{g} := g_0^{\hat{Q}_k}$. We notice that $g_2$ and $h_2$ are factors of $\tilde{g}^{Q_k}$. Since $[\tilde{2}] = \hat{\theta}_k \circ \theta_k$, the (at most) $4$ preimages of $\tilde{x}$ with respect to the map $[\tilde{2}]$ in $\Pro(\overline{\F}_q)$ are roots of $\tilde{g}^{Q_k}$. Moreover, at most one of the preimages is $[\tilde{2}]$-periodic. Therefore, without loss of generality, we can suppose that $g_2$ has a root which is not $[\tilde{2}]$-periodic. If we denote by $\tilde{x}_1$ such a root, then the base step is proved.

As regards the inductive step, suppose that $g_{2j}$ has a root $\tilde{x}_j$ belonging to the level $r+j$ of the tree $T^{q^{2^i}}_{[\tilde{2}]} (x_0)$ for some positive integers $i$ and $j$. Using the same argument as above, we can define $\tilde{g} := g_{2j}^{\hat{Q}_k}$ and notice that $g_{2j+2}$ is a factor of $\tilde{g}^{Q_k}$. Therefore, the preimages of $\tilde{x_j}$ in $T^{q^{2^{i+1}}}_{[\tilde{2}]} (x_0)$ are roots of $\tilde{g}^{Q_k}$. One of the preimages, which we denote by $\tilde{x}_{j+1}$, is a root of $g_{2j+2}$ and the inductive step is proved.
\end{enumerate}
\end{proof}

We discuss the ordinary and the supersingular case separately.

\subsection{Ordinary case: $p \equiv 1 \pmod{4}$}
Suppose that $\pi_p^{2n} - 1 = \rho_0^{e_1} \cdot \rho_1$, for some positive integer $e_1$ and some element $\rho_1 \in R$ coprime to $\rho_0$. The following holds.

\begin{theorem}\label{fin_thm_ord}
There exists a positive integer $t \leq \left\lceil \dfrac{e_1}{2} \right\rceil$ such that at least one of the following holds:
\begin{itemize}
\item $g_{t+2j-1}$ and $g_{t+2j}$ have degree $2^{1+j} \cdot n$ for any integer $j \geq 1$;
\item $h_{t+2j-1}$ and $h_{t+2j}$ have degree $2^{1+j} \cdot n$ for any integer $j \geq 1$.
\end{itemize}
\end{theorem}
\begin{proof}
Adopting the notations of Section \ref{prel_ord}, let $m=n$, $q=p^m$ and $l=0$. In accordance with Lemma \ref{lemma_seq}(2), we can say without loss of generality that, for any positive integer $j$, any polynomial $g_{2j}$ has a root belonging to the level $r+j$ of the tree $T^{q^{2^i}}_{[\tilde{2}]} (x_0)$, for some $x_0 \in \Pro (\F_q)$ and for some non-negative integer $i$. According to Theorem \ref{t_l+1}(1), the tree $T^{q^{2}}_{[\tilde{2}]} (x_0)$ has depth $\left\lceil \dfrac{e_1}{2} \right\rceil$. Let $t$ be the smallest index $2j$ such that $g_{2j}$ has a root in $\F_{q^2}$, while $g_{2j+2}$ has a root in $\F_{q^{2^2}}$. The result follows because the degree of $g_{t+2j}$ is twice the degree of $g_{t+2(j-1)}$ for any integer $j \geq 1$ and the result follows according to Theorem \ref{t_l+1}(2).
\end{proof}

\subsection{Supersingular case: $p \equiv 3 \pmod{4}$}
Suppose that $p \equiv 3 \pmod{4}$ and that $(-p)^{n} - 1 = 2^{e_0} \cdot r$, for some  integers $e_0$ and $r$. The following holds.

\begin{theorem}
There exists a positive integer $t \leq e_0$ such that at least one of the following holds:
\begin{itemize}
\item $g_{t+2j-1}$ and $g_{t+2j}$ have degree $2^{1+j} \cdot n$ for any integer $j \geq 1$;
\item $h_{t+2j-1}$ and $h_{t+2j}$ have degree $2^{1+j} \cdot n$ for any integer $j \geq 1$.
\end{itemize}
\end{theorem} 
\begin{proof}
The current theorem can be proved as Theorem \ref{fin_thm_ord} relying upon Theorem \ref{t_l+1_super} and Lemma \ref{lemma_seq}.
\end{proof}

\bibliography{Refs}
\end{document}